\documentclass[preprint,12pt]{elsarticle}
\usepackage{amsmath, amscd, amssymb, amsthm}
\usepackage{color}
\usepackage{hyperref}
\usepackage[all]{xy}
\allowdisplaybreaks

\theoremstyle{plain}
\newtheorem{theorem}{Theorem}[section]
\newtheorem{lemma}[theorem]{Lemma}
\newtheorem{proposition}[theorem]{Proposition}
\newtheorem{corollary}[theorem]{Corollary}
\newtheorem{notation}[theorem]{Notation}

\theoremstyle{definition}
\newtheorem{definition}[theorem]{Definition}
\newtheorem{remark}[theorem]{Remark}

\newtheorem{example}[theorem]{Example}
\newcommand{\N}{{\mathbb N}}     
\newcommand{\Z}{{\mathbb Z}}     
\newcommand{\CP}{{\textit{CP}}}

\newcommand{\pol}{{\textit{Pol}}}

\newcommand{\kernel}{\textit{Ker}}

\newcommand{\suc}{\textit{Suc}}

\journal{Algebra Universalis}

\begin{document}

\begin{frontmatter}

\title{Congruence Preserving Functions on   Free Monoids}

\author{Patrick C\'egielski\fnref{LACL,TARMAC}}
\ead{cegielski@u-pec.fr}

\author{Serge Grigorieff\fnref{IRIF,TARMAC}}
\ead{seg@liafa.univ-paris-diderot.fr}

\author{Ir\`ene  Guessarian%
\fnref{IRIF,emeritus,TARMAC}\corref{correspondingauthor}}
\ead{ig@liafa.univ-paris-diderot.fr}

\cortext[correspondingauthor]{Corresponding author}

\fntext[LACL]{LACL, EA 4219, Universit\'e Paris-Est Cr\'eteil,  
IUT S\'enart-Fontainebleau}

\fntext[IRIF]{IRIF, UMR 8243, CNRS \& Universit\'e Paris 7 Denis Diderot}

\fntext[emeritus]{Emeritus at UPMC Universit\'e Paris 6.}

\fntext[TARMAC]{This work was partially supported by 
TARMAC ANR agreement 12 BS02 007 01.}

\begin{abstract}
A function on an algebra  is congruence preserving
if, for any congruence, it maps congruent elements to congruent elements.
We show that, on a  free monoid generated by at least 3  letters,
a function from the free monoid into itself is  congruence preserving  
 if and only if it is   of the form $x \mapsto w_0 x w_1 \cdots w_{n-1} x w_n$
for some finite sequence of words $w_0,\ldots,w_n$. 
We generalize this result to functions of arbitrary arity.
This shows that a free monoid with at least three generators is a (noncommutative)
affine complete algebra. Up to our knowledge, 
it is the first (nontrivial) case of a noncommutative 
affine complete algebra.

\end{abstract}


\end{frontmatter}

{\scriptsize\tableofcontents}

\section{Introduction}
%
%
We here focus on functions  which are  congruence preserving  on free monoids generated by at least 3  letters.
Given an algebra $\+A=\langle A,\Omega\rangle$ 
(where $\Omega$ is a family of operations on the set $A$),
a function $f:A^k\to A$ is said to be {\em congruence preserving} 
if for every congruence $\sim$
on $\langle A,\Omega\rangle$, and for every 
$x_1,\ldots,x_k,y_1,\ldots,y_k \in A$, 
$x_1\sim y_1$,\ldots, $x_k\sim y_k$ implies $f(x_1,\ldots,x_k)\sim f(y_1,\ldots,y_k)$.
Such functions were introduced in Gr\"atzer  \cite{{gratzerUnivAlg}},  where they are said to have the ``substitution property".

 Let $O(\+A)$ be the family of all operations (of any arity) on $A$.
A {\it clone}  on $A$ is a subfamily of $O(\+A)$ containing all projections 
and closed under composition.
An important problem is to compare two clones
associated to an algebra $\langle A,\Omega\rangle$,
namely,
\begin{itemize}
\item
the smallest clone $\pol(\+A)$ which contains $\Omega$ and all constant functions
(the so-called ``polynomial functions'' by reference to the case of rings),
\item
the clone $\CP(\+A)$ of congruence preserving functions on $A$.
\end{itemize}

Obviously, we have $\pol(\+A) \subseteq \CP(\+A) \subseteq O(A)$. 
Are these inclusions strict?

In 1921 Kempner \cite{kempner1921} showed that $\pol(\+A)=O(A)$ holds for the ring $\Z/n\Z$ 
 if and only if $n$ is prime. 
More recently,  since \cite{gratzer1962}, the main concern is whether all congruence preserving functions are polynomial, i.e., 
$\pol(\+A)\stackrel{\text{?}}{=}\CP(\+A)$.
Algebras where all congruence preserving functions are polynomial are called {\em affine complete} in the terminology introduced by Werner \cite{werner1971}. They are extensively studied in the book  by Kaarli \& Pixley \cite{KaarliPixley}.

Our main results  (Theorems \ref{Carac_par_Poly} and \ref{t:main})  prove that  if $\Sigma$ has at least three elements then the free monoid $\Sigma^*$ generated by $\Sigma$  is affine complete.  Up to our knowledge, 
our result provides the first (nontrivial) case of a noncommutative 
affine complete algebra.

In the commutative case, quite a few algebras have been shown to be affine complete:
Boolean algebras (Gr\" atzer, 1962 \cite{gratzer1962}),
$p$-rings with unit (Iskander, 1972 \cite{iskander1972}),
vector spaces of dimension at least $2$ 
(Heinrich Werner, 1971 \cite{werner1971}),
free modules with more than one free generator
(N\"obauer, 1978 \cite{nobauer1978})
hence also abelian groups.
Gr\" atzer \cite{gratzer1964},  1964, determined which distributive lattices are
affine complete.
Bhargava \cite{bhargava97}, 1997,   proved that the ring $\Z/n\Z$ is affine complete 
if and only if neither $8$ nor any $p^2$ with $p$ prime divides $n$.
When $\pol(\+A)$ is a strict subfamily of $\CP(\+A)$, a natural question is
to describe the family $\CP(\+A)$.
For distributive lattices, this is done in Haviar \& Plo\v{s}\v{c}ica  \cite{plos}, 2008.

%
%
Even for such a simple arithmetical algebra as $\+A=\langle \N,\suc\rangle$
(where $\suc$ is the successor function),
the description of $\CP(\+A)$ involves nontrivial number theory.
Indeed, for the algebra $\langle \N,\suc\rangle$
a function $f:\N\to\N$ is congruence preserving if and only if 
$f(x)\geq x$ and $f$ has the following 
property:
$x-y$ divides $f(x)-f(y)$ for all $x,y\in\N$.
In~\cite{cggIJNT} we proved that this property holds 
 if and only if
\\\centerline{$f(x)=\sum_{k\in\N} a_k \dbinom{x}{k} 
= a_0+a_1 x + a_2 \dfrac{x(x-1)}{2!}+ a_3 \dfrac{x(x-1)(x-2)}{3!}+\cdots$}
where $a_k$ is divided by $\ell$ for all $2\leq \ell \leq k$.
This result also applies to the expansions of $\langle \N,\suc\rangle$
having the same congruences, e.g., 
one can expand $\langle \N,\suc\rangle$ with $+$ and $\times$.
In~\cite{cggJosef}  
we gave a similar characterization of congruence preserving functions on the algebra
$\langle \Z,+,\times\rangle$.

To give a flavor of the nontrivial character of congruence preserving functions,
let us recall some examples given in our papers \cite{cggIJNT,cggJosef}
of congruence preserving functions $f : \N\to\N$
$$
\begin{array}{ll}
f(x)\ =\ \texttt{if $x=0$ then $1$ else $\lfloor e x!\rfloor$}
&\text{($e=2,718\ldots$ is the Euler number)}
\\
f(x)\ =\ \lfloor e^{1/a} a^x   x!\rfloor
&\text{for $a\in\N\setminus\{0,1\}$}
\\
\multicolumn{2}{l}{f(x)\ =\ 
\texttt{if $x\in2\N$ then $\lfloor\cosh(1/2)\;2^x\;x!\rfloor$ 
                             else $\lfloor\sinh(1/2)\;2^x\;x!\rfloor$}}
\end{array}
$$
and of a Bessel like congruence preserving function $f : \Z\to\Z$
$$f(x)\ =\ \texttt{if $x\geq0$ then  $\dfrac{\Gamma(1/2)}{2\times4^x\times x!}
\displaystyle\int_1^\infty e^{-t/2}(t^2-1)^x dt$
else $-f(-x)$ .}
$$
It might seem counter-intuitive that, when $\Sigma^*$ has many generators, the congruence preserving functions are  fewer and much simpler than when $\Sigma^*$ has a unique generator: this  stems from the fact that, when $\Sigma^*$ has a unique generator, $\Sigma^*$  is isomorphic to $\N$ which has very few congruences, and hence a lot of functions can preserve these few congruences. 

\medskip
%
%
After recalling basic definitions in Section 2,
we prove in Section 3 that, if $\Sigma$ has at least three elements 
then the only congruence preserving functions from the free monoid $\Sigma^*$
into itself are those defined by terms with parameters, namely functions of the form
$x\mapsto w_0 x w_1 x w_2\cdots x w_n$ (Theorem \ref{Carac_par_Poly}). 
In Section 4, we extend this result by characterizing  congruence preserving functions of arbitray arity $k\in\N$ (Theorem \ref{t:main}) as the functions defined by  terms with parameters, i.e.,  functions of the form $(x_1,\ldots,x_k)\mapsto w_0x_{i_1}^{p_1}w_1x_{i_2}^{p_2}w_2\cdots x_{i_n}^{p_n}w_n$ with $x_{i_j}\in\{x_1,\ldots,x_k\}$ for $j=1,\ldots,n$. A shorter proof of Theorem \ref{t:main} is given in Section 5 when $\Sigma$ is infinite.

\section{Classical definitions and notations}\label{s:fix}
Recall the 
 notion of congruence on an algebra. 
\begin{definition}\label{def:congA}
A congruence $\sim$ on an algebra $\langle\+A,\Omega\rangle$ is an equivalence relation
on $\+A$ such that, for every operation $\xi\colon\+A^k\to\+A$ of $\Omega$, 
for all $x_1,\ldots,x_k, y_1,\ldots,y_k\in\+A $
\begin{equation*}
x_i\sim y_i \ \text  {for } i=1,\ldots,k \qquad
\Longrightarrow \qquad\xi(x_1,\ldots,x_k)\sim \xi(y_1,\ldots,y_k)
\end{equation*}
\end{definition}

\begin{definition}\label{def:cp0}
Let $\+A=\langle A,\Omega\rangle$ be an algebra  and $\sim$ a congruence on $\+A$.
A function $f : A^k \to A$ is said to {\em preserve the congruence $\sim$}  if for all $x_1,\ldots,x_{k}, y_1,\ldots,y_{k}$ in $A$,
\begin{equation*}
x_i\sim y_i \ \text  {for } i=1,\ldots,k
\quad \Longrightarrow\quad 
f(x_1,\ldots,x_{k}) \sim  f(y_1,\ldots,y_{k}) .
\end{equation*}
\end{definition}

\begin{definition}\label{def:cp}
Let $\+A=\langle A,\Omega\rangle$ be an algebra.
A function $f : A^k \to A$ is {\em congruence preserving} (abbreviated into CP) if it preserves all  congruences on $\+A$.
\end{definition}

\begin{remark}
1) A function $f$ is congruence preserving if and only if
every congruence 
on the algebra $\langle A,\Omega\rangle$
is also a congruence on the expanded algebra $\langle A,\Omega\cup\{f\}\rangle$.

2) Gr\"atzer's denomination for congruence preservation is
``{\em $f$ enjoys the substitution property}'', 
cf. 
 \cite{gratzerUnivAlg} (page 44, Chap. I \S8 below Lemma 9).
Some authors also use the denomination ``{\em $f$ is congruence compatible}''.
\end{remark}
\begin{definition} Let $\Sigma$ be an nonempty set. The {\em free monoid}   $\Sigma^*$ generated by $\Sigma$ is the  monoid $\langle\Sigma^*, \cdot\rangle$ 
\\
--  whose elements are all the finite sequences (or words) of  elements from $\Sigma$, \\
-- with  the concatenation operation: $x_1\ldots x_n\cdot  y_1\ldots y_p= x_1\ldots x_n y_1\ldots y_p$,   \\
-- whose unit element is the empty word denoted by $\varepsilon$. 
\end{definition}
For $x\in\Sigma^*$, and $n\in\N$, the word obtained by concatenating $x$ with itself $n$ times is denoted by $x^n$.

\begin{remark} The notions of CP function and morphism are different: 

(i)  $x\mapsto x^2$ is CP but is not a morphism.

(ii) Let $\Sigma=\{a,b\}$, $\varphi$ defined by $\varphi\colon a\mapsto a$, $\varphi\colon b\mapsto a$,  is a morphism but it is not CP. Indeed,  let $\sim_a$ be the congruence on $\Sigma^*$ defined by $x\sim_a y$ if and only if  $x$ and $y$ have the same number of  occurrences of $a$. 
Then $a\sim_a ab$ but $\varphi(a)\not\sim_a \varphi(ab)$.
\end{remark}

\section{Unary congruence preserving functions on  free monoids with at least three generators}
\label{s:cp Astar}
%

In an algebra any term (possibly involving elements from the algebra)
defines a congruence preserving function. We detail the case of  $\Sigma^*$  in  Lemma \ref{poly=>CP}.
\begin{lemma} \label{poly=>CP}All  unary 
functions $\Sigma^*\to \Sigma^*$ of the form $x\mapsto p(x)= w_0xw_1xw_2\ldots xw_n$,  with
$w_0,\ldots,w_n\in\Sigma^*$, are CP.
\end{lemma}

\begin{proof} Any congruence $\sim$ on $\Sigma^*$ is the kernel of some morphism $\varphi$  from $\Sigma^*$ into some monoid, 
i.e.,  $x\sim y$ if and only if  $\varphi(x)=\varphi(y)$. 
Let $\varphi$ be a morphism, and $\sim$ the associated congruence. Assume $\varphi(x)=\varphi(y)$, then 
\begin{eqnarray*}
\varphi(p(x))=\varphi(w_0xw_1xw_2\ldots xw_n)&=&\varphi(w_0)\varphi(x)\varphi(w_1)\varphi(x)\varphi(w_2)\ldots \varphi(x)\varphi(w_n)\\
&=&\varphi(w_0)\varphi(y)\varphi(w_1)\varphi(y)\varphi(w_2)\ldots \varphi(y)\varphi(w_n)\\
&=&\varphi(w_0yw_1yw_2\ldots yw_n)=\varphi(p(y))
\end{eqnarray*}
 hence 
$p(x)\sim p(y)$ and $p$ is CP.
\end{proof}
It turns out that the converse is true for CP functions $\Sigma^*\to\Sigma^*$ when $\Sigma$ has at least three letters.

\begin{theorem}\label{Carac_par_Poly}
Assume $|\Sigma|\geq 3$. 
A function $f\colon \Sigma^*\to\Sigma^*$ is CP if and only if it  is of the form $f(x)=w_0xw_1x\cdots w_{n-1}xw_n$, for some $w_0,\ldots, w_n\in\Sigma^*$.
\end{theorem}
Recall the classical relation between congruences and 
kernels of surjective homomorphisms.
\begin{proposition}\label{p:congruences}
A binary relation $\sim$ on an algebra $\langle\+A,\Omega\rangle$ is a congruence if and only if
there exists some algebra $\langle\+P,\Omega'\rangle$, where  $\Omega$ and $\Omega'$ have the same signature, a surjective homomorphism
$\theta\colon A\to P$ such that $\sim$ is the kernel $\kernel(\theta)$of $\theta$,
i.e.,  $\sim\ =\{(x,y)\mid \theta(x)=\theta(y)\}$. 
$\langle\+P,\Omega\rangle$ is isomorphic to the quotient algebra $\langle\+A,\Omega\rangle/\sim$\,.
\end{proposition}
%

We obtain  restricted notions of congruences by looking at particular monoids.
To such a restricted notion of congruence is associated an {\em a priori } enlarged notion 
of congruence preservation. 
In the next definition, we describe a particular form of restricted congruence which is crucial in our proof of Theorem~\ref{Carac_par_Poly}.

\begin{definition} 
A congruence on $\Sigma^*$ is said to be {\em restricted} if and only if it is the kernel of a morphism
$\Sigma^*\to\Sigma^*$. 
A function preserving restricted congruences is said to be {\em RCP}.
\end{definition}

\begin{example} For $a\in\Sigma$, let $\varphi \colon \Sigma^* \mapsto \langle \Z/2\Z, \times\rangle$  be the morphism  defined by $\varphi(a) =0$ and $\varphi(x) = 1$ for $x \in\Sigma\setminus\{ a\}$.  The kernel of $\varphi$ corresponds to the congruence 
``$a$ occurs in $x$ if and only if $a$ occurs in $y$". It is not a restricted congruence.\end{example}


Using the above notion, we prove a stronger version of Theorem~\ref{Carac_par_Poly}.

\begin{theorem}\label{Carac_par_Poly RCP}
Assume $|\Sigma|\geq 3$. 
A function $f\colon \Sigma^*\to\Sigma^*$ is RCP 
if and only if it  is of the form $f(x)=w_0xw_1x\cdots w_{n-1}xw_n$, for some $w_0,\ldots, w_n\in\Sigma^*$.
\end{theorem}

Theorem \ref{Carac_par_Poly RCP} has an obvious corollary which implies Theorem~\ref{Carac_par_Poly}.

\begin{corollary} Let $|\Sigma|\geq 3$.
A function $f\colon \Sigma^*\to  \Sigma^*$ is CP if and only if it is RCP.
\end{corollary}
\begin{proof}
If $f$ is CP then it is RCP. 
Theorem~\ref{Carac_par_Poly RCP} and Lemma~\ref{poly=>CP} show that 
if $f$ is RCP then it is CP.
\end{proof}

The rest of this section is devoted to the proof of Theorem \ref{Carac_par_Poly RCP},
which shall be given after Lemma~\ref{CAS_PatrickBis}.

\begin{notation} 

1) For $u$ in $\Sigma^*$ and $a\in\Sigma$,  let $|u|$ denote the length of $u$ and  let $|u|_a$ denote the number of occurrences of $a$ in $u$.

2) For $c\in\Sigma$, let $\sim_c$ be the kernel of  the morphism $\theta_c\colon \Sigma^*\to \Sigma^*$ such that $\theta_c(c)=c$, and $\theta_c(x)=\varepsilon$ for all $x\not=c$ in $\Sigma$. Thus  $u\sim_c v$ if and only if   $|u|_c=|v|_c$.
\end{notation}
\begin{lemma}\label{CPvslength} If $f$ is RCP on $\Sigma^*$ and $a,b\in \Sigma$, then 

1) $|u|=|v|$ implies $|f(u)|=|f(v)|$.

2) $|u|_a=|v|_a$ implies $|f(u)|_a=|f(v)|_a$.

3) $b\not=a$ implies  $|f(a^n)|_b=|f(\varepsilon)|_b$.
\end{lemma}

\begin{proof} 1) Let $\sim$ be 
the kernel of the  morphism $\varphi\colon \Sigma^*\to \Sigma^*$ such that $\varphi(x)=a$ for all $x$ in $\Sigma$:
$u\sim v$ if and only if   $|u|=|v|$.  As $f$ is RCP, $u\sim v$ implies $f(u)\sim f(v)$ hence $|f(u)|=|f(v)|$.

2) Similar to  the proof of 1) where $\sim$ is replaced by $\sim_a$.

3) Let $\sim_b$ be   as above, $a^n\sim_b\varepsilon$ implies $f(a^n)\sim_b f(\varepsilon)$. Thus $|f(a^n)|_b= |f(\varepsilon)|_b$.
\end{proof}
Thanks to  Lemma \ref{CPvslength}  the following notation makes sense.

\begin{notation}\label{notationEll} Let $f$ be RCP.
Denote by
by $\ell(n)$ the common length of $f(x)$ for  $x$ of length $n$ and 
by $\ell_a(n)$ the number of occurrences of the letter $a$ in the word $f(x)$ for $|x|_a=n$, i.e.,  for $x$ with  $n$ occurrences of $a$.
\end{notation}

\begin{lemma}\label{l:length}   
If $f$ is RCP on $\Sigma^*$ with $\Sigma$ containing at least two letters $a,b$, then the functions $\ell\colon \N\to\N$ 
and $\ell_a\colon \N\to\N$ defined by Notation \ref{notationEll} are affine and  of the form 
 \begin{equation*}\label{eq:f-unary ell}
\ell(n) = \big(\ell(1)-\ell(0)\big) n + \ell(0)
\quad,\quad 
\ell_a(n) = \big(\ell(1)-\ell(0) \big)n + \ell_a(0)
\end{equation*}
Thus, letting $p_f=\ell(1)-\ell(0)$ and $e_f=\ell(0)=|f(\varepsilon)|$, we have, for all $x\in\Sigma^*$, 
\begin{equation}\label{eq:f-unary}
|f(x)| = p_f |x| + e_f=p_f |x| + |f(\varepsilon)|
\quad,\quad 
|f(x)|_a = p_f |x|_a + |f(\varepsilon)|_a
\end{equation}
\end{lemma}
\begin{proof} 
1) Note that $|z| = \sum_{\alpha\in \Sigma} |z|_\alpha$. Applying this to $f(a^n)$ we get
$$
\begin{array}{rcl}
\ell(n) &=& |f(a^n)|  \\
&= &\sum_{b\in \Sigma} |f(a^n)|_b 
\\
&=& \ell_a(n) + \sum_{b\in \Sigma\setminus\{a\}} |f(\varepsilon)|_b\qquad \qquad\qquad{\text{ by Lemma \ref{CPvslength} -3)}}\\
&=&\ell_a(n) -|f(\varepsilon)|_a+ \sum_{b\in \Sigma }|f(\varepsilon)|_b=
 \ell_a(n) -\ell_a(0) +   |f(\varepsilon)|
\\
&=&\ell_a(n) -\ell_a(0) + \ell(0)
\end{array}
$$
which yields
\begin{equation}\label{ln-l0} \ell(n)-\ell(0) = \ell_a(n) - \ell_a(0)  
\end{equation}

2) Let  us now take $x=a^{n-1}b$ (recall $\Sigma$ has at least two letters).
$$
\begin{array}{rcl}
\ell(n) &=& |f(a^{n-1} b)|  =|f(a^{n-1}b)|_a + |f(a^{n-1}b)|_b + \sum_{c\neq a,b} |f(a^{n-1}b)|_c
\\
&=& \ell_a({n-1}) + \ell_b(1) + \sum_{c\neq a,b} |f(\varepsilon)|_c
\\
&=& (\ell_a({n-1})-\ell_a(0)) + (\ell_b(1)-\ell_b(0)) + \sum_{c\in \Sigma} |f(\varepsilon)|_c
\\
&=& (\ell_a({n-1})-\ell_a(0)) + (\ell_b(1)-\ell_b(0)) + |f(\varepsilon)|
\\
&=& (\ell({n-1})-\ell(0)) + (\ell(1)-\ell(0)) + \ell(0) \qquad\qquad{\text{by equation \eqref{ln-l0}}}
\\
&=& \ell({n-1}) + (\ell(1)-\ell(0)) 
\\
\ell(n) &=& n (\ell(1) - \ell(0))  + \ell(0)\qquad\qquad\qquad\qquad\qquad\quad{\text{by induction  on }}n
\end{array}
$$
Letting $p_f=\ell(1) - \ell(0)$,  we have $\ell(n)=np_f + \ell(0) = np_f+  |f(\varepsilon)|$.
Similarly, applying equation \eqref{ln-l0}, $\ell_a(n)=\ell(n) + \ell_a(0)-\ell(0) =np_f  + \ell_a(0) 
= p_f n + |f(\varepsilon)|_a$.
\end{proof}

\begin{notation}\label{sim_{a,b}}

For $c,d\in\Sigma$ let $\sim_{d,c}$ denote the congruence   kernel of the morphism identifying $d$ with $c$: $\varphi(d) =c$ and $\varphi(x) =x$ for $x\not=d$.

\end{notation}


\begin{lemma}\label{l:f epsilon} 
Assume $f : \Sigma^* \to \Sigma^*$ is RCP.
If there is some $u\neq\varepsilon$ such that $f(u)=\varepsilon$
then $f$ is constant on $\Sigma^*$ with value $\varepsilon$.
 \end{lemma}
\begin{proof}
By  equation \eqref{eq:f-unary} of Lemma~\ref{l:length} we have $|f(x)|=p_f|x|+e_f$ for all $x\in A^*$.
In particular, $0=|\varepsilon|=|f(u)|=p_f|u|+e_f$, and hence $p_f=e_f=0$.  
Thus, $|f(x)| = 0$ and $f(x) = \varepsilon$ for all $x$.
\end{proof}
We now show that
if the first letter of $f(x)$ is $b$ for {\em some letter} $x\in\Sigma$ different from $b$
then the same is true for {\em every word} $x\in\Sigma^*$.

\begin{lemma}\label{l:fa=bw} 
Assume $|\Sigma|\geq3$ and $f : \Sigma^* \to \Sigma^*$ is RCP.
If there are $a,b\in\Sigma$ such that $a\neq b$ and $f(a)\in b\Sigma^*$
then $f(x)\in b\Sigma^*$ for all words $x\in \Sigma^*$.
 \end{lemma}
\begin{proof}  We first prove that $f(c)\in b\Sigma^*$ for all $c\in\Sigma$. We argue by cases and use the morphisms  identifying letters  defined in Notation \ref{sim_{a,b}}. 
\\
\textbullet\ 
{\it Case $c\notin\{a,b\}$.}
Since $\varphi_{a,c}(a) = \varphi_{a,c}(c)$ we have $\varphi_{a,c}(f(a)) = \varphi_{a,c}(f(c))$.
As the first letter of $f(a)$ is $b$ which is not in $\{a,c\}$, 
it is equal to the first letter of $\varphi_{a,c}(f(a))$.
As $\varphi_{a,c}(f(a)) = \varphi_{a,c}(f(c))$, the first letter of $\varphi_{a,c}(f(c))$ is $b$ 
hence the first letter of $f(c)$ must also be $b$. 
Thus, $f(c)\in b\Sigma^*$.\\
\textbullet\ 
{\it Case $c=a$.} Trivial since condition $f(a)\in b\Sigma^*$ is our assumption.
\\
\textbullet\ 
{\it Case $c=b$.}
We know (by the two previous cases) that
$f(x)\in b\Sigma^*$ for all $x\in A\setminus\{b\}$.
As $|\Sigma|\geq 3$ there exists $d\notin\{a,b\}$.
Observe that
\\
- $\varphi_{a,b}(a) = \varphi_{a,b}(b)$, and  hence
$\varphi_{a,b}(f(a)) = \varphi_{a,b}(f(b))$.
As $f(a)\in b\Sigma^*$, we get $f(b)\in \{a,b\} \Sigma^*$.
\\
- $\varphi_{d,b}(d) = \varphi_{d,b}(b)$ whence
$\varphi_{d,b}(f(d)) = \varphi_{d,b}(f(d))$.
As $f(d)\in b\Sigma^*$, we get $f(b)\in \{d,b\} \Sigma^*$.
\\
Thus, $f(b)\in \{a,b\}\Sigma^* \cap \{d,b\} \Sigma^* = b\Sigma^*$.

\medskip
 We next prove by induction on $n=|x|$ 
that $f(x)\in b\Sigma^*$ for all words $x\in \Sigma^*$.\\
\textbullet\
{\it Base case $n=1$:} The  case $n=1$  coincides with what was proved above.
\\\textbullet\ 
{\it Base case $n=0$:} 
Since $|\Sigma|\geq3$ there are  $c,d\in\Sigma$ such that $b,c,d$ are pairwise distinct.
The base case $n=1$ insures that $f(c)=bu$ and $f(d)=bv$ for some $u,v\in\Sigma^*$.
For $c\in\Sigma$ let $\psi_c : \Sigma^*\to\Sigma^*$ be the morphism which erases $c$~:
$\psi_c(c)=\varepsilon$ and $\psi_c(x)=x$ for every $x\in\Sigma\setminus\{c\}$.
As $\psi_c(\varepsilon) = \psi_c(c)$ we have  
$\psi_c(f(\varepsilon)) = \psi_c(f(c)) = \psi_c(bu) = bs$ for some $s\in\Sigma^*$.
Equality $\psi_c(f(\varepsilon)) = bs$ shows that $f(\varepsilon) \in c^*b\Sigma^*$.
Arguing with $\psi_d$ we similarly get $f(\varepsilon) \in d^*b\Sigma^*$.
Thus, $f(\varepsilon)
 \in c^*b\Sigma^*\cap d^*b\Sigma^* = b\Sigma^*$.
\\\textbullet\
{\it Inductive step: from $\leq n$ to $n+1$ where $n\geq1$.}
We assume that $f(y)\in b\Sigma^*$ for every $y\in\Sigma^*$ of length at most $n$.
Let $x\in\Sigma^{n+1}$, we prove that $f(x)\in b\Sigma^*$.
We argue by cases.
\\
$\blacktriangleright$\
{\it Case 1: $|\{c\in\Sigma\setminus\{b\} \mid c\text{ occurs in }x\}| \geq 2$.}
Then $x=ucvdw$ where $u,v,w\in\Sigma^*$ and $c,d,b$ are pairwise distinct letters in $\Sigma$.
We consider the erasing morphisms $\psi_c$ and $\psi_d$.
As $|uvdw|=n$ the induction hypothesis yields $f(uvdw)=bt$ for some $t\in\Sigma^*$.
Now, $\psi_c(x)=\psi_c(uvdw)$, and hence 
$\psi_c(f(x))=\psi_c(f(uvdw))=\psi_c(bt)=bs$ for some $s\in\Sigma^*$.
Equality $\psi_c(f(x))=bs$ shows that $f(x)\in c^*b\Sigma^*$.
Arguing similarly with $\psi_d$ and $ucvw$, we get $f(x)\in d^*b\Sigma^*$.
Thus, $f(x)\in c^*b\Sigma^* \cap d^*b\Sigma^* = b\Sigma^*$.
\\
$\blacktriangleright$\
{\it Case 2: $c$ is the unique letter in $\Sigma\setminus\{b\}$ which occurs in $x$
and it occurs at least twice.}
Then $x=ucvcw$ where $u,v,w\in\Sigma^*$.
As $|\Sigma|\geq 3$ there exists $d\notin\{b,c\}$.
The word $ucvdw$ is relevant to Case 1, and hence $f(ucvdw)=bt$ for some $t\in\Sigma^*$.
We consider the morphism $\varphi_{d,c}$ which identifies $d$ with $c$.
We have $\varphi_{d,c}(ucvdw)=x=\varphi_{d,c}(x)$. Hence 
$\varphi_{d,c}(f(x))=\varphi_{d,c}(f(ucvdw))=\varphi_{d,c}(bt)=bs$
for some $s\in\Sigma^*$.
Since $b\notin\{c,d\}$ equality $\varphi_{d,c}(f(x))=bs$ shows that $f(x)\in b\Sigma^*$.
\\
$\blacktriangleright$\
{\it Case 3: $c$ is the unique letter in $\Sigma\setminus\{b\}$ which occurs in $x$
and it occurs only once.}
Then $x=b^k c b^\ell$ where $k+\ell=n\geq1$.
The word $c^{n+1}$ is relevant to Case 2, and hence $f(c^{n+1})=bt$ for some $t\in\Sigma^*$.
Consider the morphism $\varphi_{b,c}$ which identifies $b$ with $c$.
We have $\varphi_{b,c}(x)=\varphi_{b,c}(c^{n+1})$. Hence 
$\varphi_{b,c}(f(x))=\varphi_{b,c}(f(c^{n+1}))=\varphi_{b,c}(bt)=bs$
for some $s\in\Sigma^*$.
Equality $\varphi_{b,c}(f(x))=bs$ shows that $f(x)\in \{b,c\}\Sigma^*$.

As $|\Sigma|\geq 3$ there exists $d\notin\{b,c\}$.
Let $y$ be obtained from $x$ by replacing the first occurrence of $b$ by $d$.
The word $y$ is relevant to Case 1, and hence $f(y)=bt$ for some $t\in\Sigma^*$. 
Consider the morphism $\varphi_{d,b}$ which identifies $d$ with $b$.
We have $\varphi_{d,b}(y)=x=\varphi_{d,b}(x)$. Hence 
$\varphi_{d,b}(f(x))=\varphi_{d,b}(f(y))=\varphi_{d,b}(bt)=bs$ for some $s\in\Sigma^*$.
Equality $\varphi_{d,b}(f(x))=bs$ shows that $f(x)\in \{b,d\}\Sigma^*$.
Thus, $f(x)\in \{b,c\}\Sigma^* \cap \{b,d\}\Sigma^* = b\Sigma^*$.
\\
$\blacktriangleright$\
{\it Case 4: $x=b^{n+1}$.}
As $|\Sigma|\geq 3$ there exists $c,d\in\Sigma$ such that $b,c,d$ are pairwise distinct.
The words $b^nc$ and $b^nd$ are relevant to Case 3, and
hence $f(b^nc)=bt$ and $f(b^nd)=bs$ for some $s,t\in\Sigma^*$.  
Consider the morphism $\varphi_{c,b}$ which identifies $c$ with $b$.
We have $\varphi_{c,b}(b^nc)=x=\varphi_{c,b}(x)$. Hence 
$\varphi_{c,b}(f(x))=\varphi_{c,b}(f(b^nc))=\varphi_{c,b}(bt)=br$ for some $r\in\Sigma^*$,
whence $f(x)\in \{b,c\}\Sigma^*$.
Arguing similarly with $b^nd$ and the morphism $\varphi_{d,b}$ which identifies $d$ with $b$
we get $f(x)\in \{b,d\}\Sigma^*$.
Thus, $f(x)\in \{b,c\}\Sigma^* \cap \{b,d\}\Sigma^* = b\Sigma^*$.
\end{proof}
We now show that
if $x$ is a prefix of $f(x)$ for every letter $x\in\Sigma$ 
then the same is true for every word $x\in\Sigma^*$.

\begin{lemma}\label{l:fx=xw} 
Assume $|\Sigma|\geq3$ and $f : \Sigma^* \to \Sigma^*$ is RCP.
If $f(a)\in a\Sigma^*$ for all $a\in\Sigma$
then $f(x)\in x\Sigma^*$ for all $x\in\Sigma^*$.
 \end{lemma}
\begin{proof}
We argue by induction on the length of $x$.
\\
{\it Base case $|x|=0$.} Condition $f(\varepsilon)\in \varepsilon\Sigma^*$ is trivial.
\\
{\it Base case $|x|=1$.} This is our assumption.
\\
{\it Inductive step: from $n\geq1$ to $n+1$.}
Assuming $f(x_1\cdots x_n)\in x_1\cdots x_n \Sigma^*$ for all
$x_1,\ldots,x_n\in\Sigma$, and letting $x_{n+1}=b$, 
we prove $f(x_1\cdots x_n b)\in x_1\cdots x_n b\Sigma^*$
for all $x_1,\ldots,x_n,b\in\Sigma$.
Claims 1 and 2 below respectively deal with the cases $b\neq x_n$ and $b=x_n$. 
\smallskip
\\
{\bf Claim 1.} If $b\neq x_n$ then $f(x_1\ldots x_{n-1}x_nb) \in x_1\ldots x_{n-1}x_nb\Sigma^*$.
\smallskip
\\
{\it Proof of Claim 1.} Let $x_1\cdots x_{n-1} = x_n^{\ell_0}  y_1  x_n^{\ell_1} \cdots y_p  x_n^{\ell_p}$
with $y_1,\ldots,y_p\in\Sigma\setminus\{x_n\}$ and $p,\ell_0,\ldots, \ell_p \in\N$.
If $b\neq x_n$ then we show that  there are $k_0,\ldots, k_p \in\N$ and $w\in \Sigma^*$ such that
\begin{equation}\label{eq:claim1}
f(x_1\cdots x_{n-1}x_nb)=f(x_n^{\ell_0}  y_1  x_n^{\ell_1} \cdots y_p  x_n^{\ell_p})
= x_n^{k_0} y_1  x_n^{k_1} \cdots y_{p-1}  x_n^{k_{p-1}} y_p  x_n^{k_p} b w.
\end{equation}
Indeed, since $|x_1\cdots x_{n-1}b|=n$ we have
$f(x_1\cdots x_{n-1}b) \in x_1\cdots x_{n-1}b \Sigma^*$ by  the induction hypothesis.
Thus, $f(x_1\cdots x_{n-1}b) = x_1\cdots x_{n-1}bu
 = x_n^{\ell_0}  y_1  x_n^{\ell_1} \cdots y_p  x_n^{\ell_p}  b u$ for some $u\in\Sigma^*$.
Let $\psi : \Sigma^*\to\Sigma^*$ be the morphism which erases $x_n$~:
$\psi(x_n)=\varepsilon$ and $\psi(y)=y$ for $y\in\Sigma\setminus\{x_n\}$. 
We have $\psi(x_1\ldots x_{n-1}x_nb) = \psi(x_1\ldots x_{n-1}b)$ 
hence 
$\psi(f(x_1\ldots x_{n-1}x_nb)) = \psi(f(x_1\ldots x_{n-1}b))
= \psi(x_n^{\ell_0}  y_1  x_n^{\ell_1} \cdots y_p  x_n^{\ell_p}  b u)
= y_1 y_2\dots y_p  b  \psi(u).$
Thus, $f(x_1\ldots x_{n-1}x_nb)$ is obtained by inserting some occurrences of $x_n$ 
in $y_1 y_2\cdots y_p  b  \psi(u)$, and hence equation \eqref{eq:claim1} holds.

As $|\Sigma|\geq 3$ there is $c$ such that $c\not\in\{b,x_n\}$. 
Let $ \varphi : \Sigma^*\to\Sigma^*$ be the morphism such that 
$ \varphi(b)=bc$,  $ \varphi(c)=x_nbc $  and $ \varphi(y)=y$ for $y\in\Sigma\setminus\{b,c\}$. 
We have $ \varphi(x_nb)= \varphi(c)$, and hence 
 $ \varphi(x_1\ldots x_{n-1} x_n b) =  \varphi(x_1\cdots x_{n-1}c)$.
Thus  $ \varphi(f(x_1\cdots x_{n-1}x_nb)) =  \varphi(f(x_1\cdots x_{n-1}c))$. 
\\
Applying $\varphi$ to equation \eqref{eq:claim1}
\begin{eqnarray}\notag 
\varphi(f(x_1\ldots x_{n-1}x_nb))
&= & \varphi(x_n^{k_0}  y_1   x_n^{k_1} \cdots y_{p-1}   x_n^{k_{p-1}}  y_p  x_n^{k_p}  b  w)\\ \label{eq:b}
&=&x_n^{k_0}   \varphi(y_1)  x_n^{k_1} \cdots \varphi(y_p)  x_n^{k_p}  bc   \varphi(w) 
\end{eqnarray}
As $|x_1\ldots x_{n-1}c|=n$,  by the induction hypothesis for length $n$, there is $u\in\Sigma^*$ such that
$$ f(x_1\ldots x_{n-1}c) = x_1\cdots x_{n-1}cu=x_n^{\ell_0}  y_1  x_n^{\ell_1} \cdots y_p  x_n^{\ell_p}  cu.$$
Hence
\begin{equation}\label{eq:c}
\varphi(f(x_1\ldots x_{n-1}c)) =  \varphi(x_n^{\ell_0}  y_1  x_n^{\ell_1} \cdots y_p  x_n^{\ell_p}  cu)
= x_n^{\ell_0}   \varphi(y_1)  x_n^{\ell_1} \cdots  \varphi(y_p)  x_n^{\ell_p}  x_nbc   \varphi(u)
\end{equation}
As $ \varphi(f(x_1\cdots x_{n-1}x_nb)) =  \varphi(f(x_1\ldots x_{n-1}c))$,
we infer from Equations \eqref{eq:b} and \eqref{eq:c}
$$x_n^{k_0}   \varphi(y_1)  x_n^{k_1}    \varphi(y_2)  x_n^{k_2}
\cdots \varphi(y_p)  x_n^{k_p}  bc   \varphi(w)
= x_n^{\ell_0}   \varphi(y_1)  x_n^{\ell_1}  \varphi(y_2)  x_n^{\ell_2}  \cdots 
 \varphi(y_p)  x_n^{\ell_p}  x_nbc   \varphi(u)
$$
Since $y_1\neq x_n$ the word $ \varphi(y_1)$ contains a letter distinct from $x_n$
and the last equality implies $k_0=\ell_0$. Thus
$x_n^{k_1}    \varphi(y_2)  x_n^{k_2} \cdots \varphi(y_p)  x_n^{k_p}  bc   \varphi(w)
=  x_n^{\ell_1}  \varphi(y_2)  x_n^{\ell_2}  \cdots 
 \varphi(y_p)  x_n^{\ell_p}  x_nbc   \varphi(u)$.
Iterating the previous argument we get $k_i=\ell_i$ for $i=0,\ldots, p-1$ and  
the residual equality
$x_n^{k_p}  bc   \varphi(w) =  x_n^{\ell_p}  x_nbc   \varphi(u)$
This last equality implies $k_p=\ell_p+1$.
Thus, 
\begin{eqnarray*}
f(x_1\ldots x_{n-1}x_nb)
&=& x_n^{k_0}  y_1   x_n^{k_1} \cdots y_{p-1}   x_n^{k_{p-1}}  y_p  x_n^{k_p}  b  w
\\
&=& x_n^{\ell_0}  y_1   x_n^{\ell_1} \cdots 
y_{p-1}   x_n^{\ell_{p-1}}  y_p  x_n^{\ell_p} x_n  b  w
\\
&=& x_1\cdots x_{n-1} x_n  b  w
\end{eqnarray*}
\smallskip
{\bf Claim 2.} {\it If $b=x_n$ then 
$f(x_1\ldots x_{n-1}bb) \in x_1\ldots x_{n-1}bb\Sigma^*$.}
\smallskip
\\
{\it Proof of Claim 2.}
Let $a\in\Sigma\setminus\{b\}$ 
and consider the morphism $\varphi_{a,b}$ which identifies $a$ with $b$~:
$\varphi_{a,b}(a)=b$ and $\varphi_{a,b}(x)=x$ for $x\in\Sigma\setminus\{a\}$. 
As $\varphi_{a,b}(x_1\ldots x_{n-1}bb) = \varphi_{a,b}(x_1\ldots x_{n-1}ba)$
we have 
$\varphi_{a,b}(f(x_1\ldots x_{n-1}bb)) = \varphi_{a,b}(f(x_1\ldots x_{n-1}ba))$.
As $a\neq b$, we can apply Claim~1:
$f(x_1\ldots x_{n-1}ba)
= x_1\ldots x_{n-1}ba w$ for some $w\in\Sigma^*$.
Thus,
\begin{eqnarray*}
\varphi_{a,b}(f(x_1\ldots x_{n-1}bb)) &=& \varphi_{a,b}(x_1\ldots x_{n-1}baw)
\ =\ \varphi_{a,b}(x_1\ldots x_{n-1})bb \varphi_{a,b}(w),\\
\text{and hence\qquad}
f(x_1\ldots x_{n-1}bb) &\in& 
\varphi_{a,b}^{-1}(\varphi_{a,b}(x_1\ldots x_{n-1})) \{aa,ab,ba,bb\}\Sigma^*
\end{eqnarray*}
Similarly, let $c\not\in\{a,b\}$.
Considering the morphism $\varphi_{c,b}$ 
which identifies $c$ with $b$~:
$\varphi_{c,b}(c)=b$ and $\varphi_{c,b}(x)=x$ for $x\in\Sigma\setminus\{a\}$,
we get
\begin{eqnarray*}
f(x_1\ldots x_{n-1}bb) &\in& 
\varphi_{c,b}^{-1}(\varphi_{c,b}(x_1\ldots x_{n-1})) \{cc,cb,bc,bb\}\Sigma^*
\end{eqnarray*}
Let $z=z_1\ldots z_{n-1}$ be the length $n-1$ prefix of $f(x_1\ldots x_{n-1}bb)$.
As $\varphi_{a,b}$ and  $\varphi_{c,b}$ preserve length,
$z\in \varphi_{a,b}^{-1}(\varphi_{a,b}(x_1\ldots x_{n-1}))
\cap\varphi_{c,b}^{-1}(\varphi_{c,b}(x_1\ldots x_{n-1}))$.
Since $z\in \varphi_{a,b}^{-1}(\varphi_{a,b}(x_1\ldots x_{n-1}))$ 
we have $z_i\in\{a,b\} \Leftrightarrow x_i\in\{a,b\}$, and $z_i\notin\{a,b\} \Rightarrow z_i=x_i$.
Similarly, $z\in \varphi_{c,b}^{-1}(\varphi_{c,b}(x_1\ldots x_{n-1}))$ implies that
 $z_i\in\{c,b\} \Leftrightarrow x_i\in\{c,b\}$ and $z_i\notin\{c,b\} \Rightarrow z_i=x_i$.
Thus,
\\- if $z_i=b$ then both $z_i$ and $x_i$ are in $\{a,b\}\cap\{c,b\}=\{b\}$, and $z_i=x_i=b$,
\\- if $z_i\neq b$ then either $z_i\notin\{a,b\}$ or $z_i\notin\{c,b\}$, and in both cases $z_i=x_i$.
\\
This proves that $z=x_1\ldots x_{n-1}$.

Finally, the two letters of $f(x_1\ldots x_{n-1}bb)$ which follow $z$ are in
$\{aa,ab,ba,bb\}\cap\{cc,cb,bc,bb\} =\{bb\}$.
Thus, $f(x_1\ldots x_{n-1}bb) \in x_1\ldots x_{n-1}bb\Sigma^*$ and Claim 2 is proved,
 finishing the proof of the Lemma.
\end{proof}

As a corollary of Lemmata \ref{l:f epsilon},   \ref{l:fa=bw} and  \ref{l:fx=xw}, we get

\begin{lemma}\label{Patrick} 
Assume $|\Sigma|\geq 3$ and $f\colon \Sigma^*\to\Sigma^*$ is RCP.  Exactly one of the below three conditions holds: 
\begin{enumerate}
\item [($C_1$)]
 either  there exists  $b\in\Sigma$ such that $f(x) \in b\Sigma^*$ for all  $x\in \Sigma^*$,
\item [($C_2$)] or $f(x) \in x\Sigma^*$ for all  $x\in \Sigma^*$,
\item [($C_3$)] or  $f$ is constant on $\Sigma^*$ with value $\varepsilon$. 
\end{enumerate}
 \end{lemma}
 The proof of Theorem \ref{Carac_par_Poly RCP} relies on the property of RCP functions which is stated in the next Lemma \ref{CAS_PatrickBis}.
\begin{lemma}\label{CAS_PatrickBis} Assume $|\Sigma|\geq 3$ and $f\colon \Sigma^*\to\Sigma^*$ is RCP and such that either $f(x)=ag(x)$ for all $x\in \Sigma^*$, or  $f(x)=x g(x)$  for all $x\in \Sigma^*$.  
Then $g$ is also RCP.
  \end{lemma}

\begin{proof} For $\varphi\colon \Sigma^*\to  \Sigma^*$ a  morphism, $\varphi(x)=\varphi(y)$ implies $\varphi(f(x))=\varphi(f(y))$.
If $f(x)=ag(x)$  for all $z\in\Sigma^*$ then  $\varphi(a)\varphi(g(x))=\varphi(a)\varphi(g(y))$. Cancelling the common prefix  $\varphi(a)$, we get  $\varphi(g(x))=\varphi(g(y))$. Similarly, if $f(x)=xg(x)$  for all $z\in\Sigma^*$, we conclude by cancelling the common prefix  $\varphi(x)$.
\end{proof}


\begin{proof} [Proof of  Theorem \ref{Carac_par_Poly RCP}] 
The sufficient condition follows from Lemma \ref{poly=>CP}.\\
We now prove the necessary condition:
\begin{equation}\label{nec-cn-RCP}
\text{ if $f$ is RCP then it is of the form $f(x)=w_0xw_1x\cdots w_{n-1}xw_n$.}
\end{equation}
We argue by induction on  $p_f +e_f$, where $p_f ,e_f$ are defined in Lemma \ref{l:length}. 

\noindent 
{\textbullet} {\it Basis:} If $p_f +e_f=0$, then $p_f =e_f=0$ and $f(x)=\varepsilon$ which is of the required form with $n=0$  and
$w_0=\varepsilon$.

\noindent 
{\textbullet} {\it Induction:} Assume that condition \eqref{nec-cn-RCP}  holds for every function $h$ such that $p_h+e_h \leq n$
and let $f$ be RCP with $p_f+e_f = n+1$. We first note that, as $p_f +e_f\geq 1$, $f$ cannot be the constant function with value $\varepsilon$. 
Hence, by Lemma \ref{Patrick}, $f$ is of the form 
$f(x)=ag(x)$ or $f(x)=xg(x)$. Moreover, by Lemma~\ref{CAS_PatrickBis}, $g$ is RCP. \\
-- If $f(x)=ag(x)$ for all $x\in\Sigma^*$ then $|f(x)|=1+|g(x)|$, 
 $p_f=p_g$ and $e_f=e_g+1$.\\
-- If $f(x)=xg(x)$ for all $x\in\Sigma^*$ then $|f(x)|=|x|+|g(x)|$, 
 $p_f=p_g+1$ and $e_f=e_g$.\\
Hence in both cases
$p_g+e_g=p_f +e_f-1$. Thus,  by the induction hypothesis, $g$ is of the required form and so is $f$.
\end{proof}
%
%

\section{Non unary congruence preserving functions on  free monoids with at least three generators}
In the present Section  we   extend Theorem  \ref{Carac_par_Poly} and  characterize CP functions $f\colon (\Sigma^*)^k\to\Sigma^*$ of arbitrary arity $k$ (cf. Theorem \ref{t:main}).
To this end we use the notion of {\em RCP $k$-ary function}.
The idea of the proof is similar to the idea of the proof of Theorem  \ref{Carac_par_Poly}.

\begin{definition} 
A function $f\colon (\Sigma^*)^k\to\Sigma^*$ preserving restricted congruences is said to be {\em RCP}.
\end{definition}


\begin{lemma}\label{CPvslengthK} 
Let $k\geq1$, $f : (\Sigma^*)^k \to \Sigma^*$ be RCP,
$u_1,\ldots, u_k, v_1,\dots, v_k \in\Sigma^*$,  $a,b\in\Sigma$
and $n_1,\ldots,n_k\in\N$.
\\
1. $ |u_i|=|v_i|$ for all $i\in\{1,\ldots,k\}$
implies $|f(u_1,\ldots,u_k)|=|f(v_1,\ldots,v_k)|$.
\\
2. $ |u_i|_a=|v_i|_a$ for all $i\in\{1,\ldots,k\}$
implies $|f(u_1,\ldots,u_k)|_a=|f(v_1,\ldots,v_k)|_a$.
\\
3. $b\neq a$ implies $|f(a^{n_1},\ldots, a^{n_k})|_b=|f(\varepsilon,\ldots,\varepsilon)|_b$.
\end{lemma}
\begin{proof} Similar to the proof of Lemma \ref{CPvslength}.
%
%
\end{proof}
Thanks to Lemma \ref{CPvslengthK}  the following notations make sense.

\begin{notation}\label{not:ell}
Let $k\geq1$. Denote by $\vec{0}$  the $k$-tuple $(0,\ldots,0)$ and by 
$\vec{e_1},\ldots, \vec{e_k}$ respectively the $k$-tuples $(1,0,\ldots,0)$, \ldots, $(0,\ldots,0,1)$.

 If $k\geq1$, $f : (\Sigma^*)^k \to \Sigma^*$ is RCP and $|x_i|=n_i$ for $i=1,\ldots,k$, let
$$\begin{array}{rcl}
\ell(n_1,\ldots,n_k) &=&\!\!\!\!  \text{common value of $|f(x_1,\ldots,x_k)|$ with  
$(x_1,\ldots, x_k)\in\Sigma^{n_1}\!\times\cdots\times\Sigma^{n_k}$}
\\
\ell_a(n_1,\ldots,n_k) &=&\!\! \!\!  \text{common value of $|f(x_1,\ldots,x_k)|_a$ with  
$|x_1|_a=n_1, \dots, |x_k|_a=n_k$}
\\
\Delta\ell(n_1,\ldots,n_k) &=& \ell(n_1,\ldots,n_k) - \ell(\vec{0}) =  \ell(n_1,\ldots,n_k)-|f(\varepsilon,\ldots,\varepsilon)|
\\
\Delta\ell_a(n_1,\ldots,n_k) &=& \ell_a(n_1,\ldots,n_k) - \ell_a(\vec{0})  
\end{array}$$
\end{notation}

\begin{lemma}\label{l:lengthK}   
If $\Sigma$ contains at least two letters
and $f : (\Sigma^*)^k \to \Sigma^*$ is RCP
then there exist $p_{f,1}, \ldots,p_{f,k},e_f$ in $\N$ such that
\begin{eqnarray}
\label{eq:length-f} |f(x_1,\ldots,x_k)| &=& p_{f,1} |x_1| + \cdots + p_{f,k} |x_k| + e_f\\ 
\label{eq:length-f_a} \text{for all $a\in\Sigma$\qquad}
|f(x_1,\ldots,x_k)|_a &=& p_{f,1} |x_1|_a + \cdots + p_{f,k} |x_k|_a 
+ |f(\varepsilon,\ldots,\varepsilon)|_a\qquad\qquad
\end{eqnarray}
with
$p_{f,i}=\ell(\vec{e_i}) - \ell(\vec{0})
= |f(\overbrace{\varepsilon,\ldots,\varepsilon}^{\text{$(i-1)$ times}}, a, \varepsilon,\ldots,\varepsilon)|
 - |f(\varepsilon,\ldots,\varepsilon)|$ and $e_f=|f(\varepsilon,\ldots,\varepsilon)|$.
 
\end{lemma}
\begin{proof} 
Observe that $|z| = \sum_{c\in \Sigma} |z|_c$. 
Using Notation~\ref{not:ell}, we have
\begin{eqnarray}\notag
\ell(n_1,\ldots,n_k) &=& |f(a^{n_1},\ldots, a^{n_k})|
\ =\ \sum_{b\in \Sigma} |f(a^{n_1},\ldots, a^{n_k})|_b 
\\\notag
&=& |f(a^{n_1},\ldots, a^{n_k})|_a 
+ \sum_{b\in \Sigma\setminus\{a\}} |f(\varepsilon,\ldots, \varepsilon)|_b
\quad{\text{ by Lemma \ref{CPvslengthK}-3)}}
\\\notag
&=& \big(|f(a^{n_1},\ldots, a^{n_k})|_a - |f(\varepsilon,\ldots, \varepsilon)|_a\big)
+ \sum_{b\in \Sigma} |f(\varepsilon,\ldots, \varepsilon)|_b
\\\notag
&=& \big(|f(a^{n_1},\ldots, a^{n_k})|_a - |f(\varepsilon,\ldots, \varepsilon)|_a\big)
+ |f(\varepsilon,\ldots, \varepsilon)|
\\\notag
&=& \ell_a(n_1,\ldots,n_k) -  \ell_a(\vec{0}) + \ell(\vec{0})
\quad{\text{ hence}}
\\\label{eq:Delta ell = Delta ellea}
\Delta\ell(n_1,\ldots,n_k) &=& \Delta \ell_a(n_1,\ldots,n_k) 
\end{eqnarray}
Now, let  $\vec x=(x_1,x_2,\ldots,x_k) = (a^{n_1-1}b, a^{n_2},\ldots, a^{n_k})$ 
with $a\neq b$ (possible as $\Sigma$ has at least two letters).
Then
\begin{eqnarray}\notag
\ell(n_1,\ldots,n_k) &=& |f(a^{n_1-1} b, a^{n_2},\ldots, a^{n_k})|  
\\\notag
&=& |f(\vec x)|_a + |f(\vec x)|_b + \sum_{c\neq a,b} |f(\vec x)|_c
\\\notag
&=& \ell_a(n_1-1,n_2, \ldots,n_k) + \ell_b(\vec{e_1}) +  \sum_{c\neq a,b} \ell_c(\vec{0}) 
\\\notag
&=& \Delta\ell_a(n_1-1,n_2, \ldots,n_k) + \Delta\ell_b(\vec{e_1})
+ \sum_{c\in\Sigma} \ell_c(\vec{0}) 
\\\notag
&=& \Delta\ell_a(n_1-1,n_2, \ldots,n_k) + \Delta\ell_b(\vec{e_1}) + \ell(\vec{0}) 
\\\notag
\text{hence:}\qquad\qquad \Delta\ell(n_1,\ldots,n_k)
&=& \Delta\ell_a(n_1-1,n_2, \ldots,n_k) + \Delta\ell_b(\vec{e_1}) 
\\\notag
\text{Using \eqref{eq:Delta ell = Delta ellea}, \ \ }\quad\quad \Delta\ell(n_1,\ldots,n_k)
&=& \Delta\ell(n_1-1,n_2, \ldots,n_k) + \Delta\ell(\vec{e_1}) 
\\\label{eq:n_1}
\text{Iterating, \quad \ } \quad\quad \Delta\ell(n_1,\ldots,n_k)
&=&  \Delta\ell(0,n_2, \ldots,n_k) + n_1 \Delta\ell(\vec{e_1})
\\\label{eq:n_2}
\text{Similarly,\ \ }\quad\quad  \Delta\ell(0,n_2, \ldots,n_k)
&=&  \Delta\ell(0,0,n_3, \ldots,n_k) + n_2  \Delta\ell(\vec{e_2}) 
\\\notag
&\vdots&
\\\label{eq:n_k}
\Delta\ell(0,0, \ldots,0,n_k)
&=&  \Delta\ell(\vec{0}) + n_k  \Delta\ell(\vec{e_k})
\ =\ n_k  \Delta\ell(\vec{e_k})
\\\notag
\text{Summing lines \eqref{eq:n_1} to \eqref{eq:n_k} gives :\ \ }&&
\\\label{eq:Delta ell}
 \Delta\ell(n_1,\ldots,n_k)
&=&{\sum_{i=1}^{i=k} }n_i \Delta\ell(\vec{e_i}) 
\end{eqnarray}
Equality \eqref{eq:Delta ell} together with Lemma \ref{CPvslengthK} implies equation \eqref{eq:length-f}. Equation 
 \eqref{eq:length-f_a} then follows from equation \eqref{eq:Delta ell = Delta ellea}.
\end{proof}

We use the characterization of unary CP functions $f\colon \Sigma^*\to\Sigma^*$ given in  Theorem \ref{Carac_par_Poly} to  characterize  $k$-ary CP functions $f\colon (\Sigma^*)^k\to\Sigma^*$ in Theorem \ref{t:main}. The key result for proving this characterization is Lemma \ref{PatrickInd}, which extends Lemma \ref{Patrick} to arity $k\geq 2$. 

\begin{lemma}\label{PatrickInd} 
Assume $|\Sigma|\geq 3$ and let  $f\colon (\Sigma^*)^k\to\Sigma^*$ be RCP.  Exactly one of the below three conditions holds 
\begin{enumerate}
\item [($C_1$)]
either there exists  $b\in\Sigma$ such that $f(x_1,\dots,x_k) \in b\Sigma^*$ for all  $x_1,\dots,x_k\in \Sigma^*$,
\item  [($C_2$)]
or there exists $i\in\{1,\ldots, k\}$ such that $f(x_1,\dots,x_k) \in x_i\Sigma^*$ for all  $x_1,\dots,x_k\in \Sigma^*$,
\item  [($C_3$)] or $f(x_1,\dots,x_k)$ is  constant with value $\varepsilon$ for all $x_1,\dots,x_k$.
 \end{enumerate}
 \end{lemma}
 Before proving Lemma \ref{PatrickInd}, we first show how  Lemma \ref{PatrickInd} together with Lemma \ref{l:lengthK}  entails  the following characterization of $k$-ary CP functions, extending Theorem \ref{Carac_par_Poly}. 
\begin{theorem}\label{t:main} Let   $|\Sigma|\geq 3$.  A function $f\colon (\Sigma^*)^k\to\Sigma^*$  is CP if and only if 
there exist $n\in\N$, $w_0,\ldots,w_n\in\Sigma^*$ and $p_1,\ldots,p_n\in\N$ such that for all
 $x_1,\ldots,x_k\in\Sigma^*$,
$f(x_1,\ldots,x_k)=w_0x_{i_1}^{p_1}w_1x_{i_2}^{p_2}w_2\cdots x_{i_n}^{p_n}w_n$,
where $x_{i_j}\in\{x_1,\ldots,x_k\}$ for $j=1,\ldots,n$.\end{theorem}
%


\begin{proof}
 The ``if" part (sufficient condition)  is clear as in the unary functions case. 
 
  For the ``only if" part we first note that if $f$ is CP then it is RCP. We next prove 
  that if $f$ is RCP then  it is of the form stated in the Theorem. The proof is by induction on 
$p_{f,1} + p_{f,2}+ \cdots+p_{f,k}+e_f $ where  $|f(x_1,\ldots,x_k)| 
= p_{f,1} |x_1|+ p_{f,2} |x_2|+ \cdots+p_{f,k} |x_k|+e_f$
(cf. Lemma \ref{l:lengthK}).

{\it Basis:}\/ if $p_{f,1} + p_{f,2}+ \cdots+p_{f,k}+ e_f =0$ then $p_{f,1}=\cdots=p_{f,k}=e_f=0$ and $f(x_1,\ldots,x_k)=\varepsilon$ which is of the required form
with $n=0$ and $w_0=\varepsilon$.

{\it Induction: }\/Otherwise, $p_{f,1} + p_{f,2}+ \cdots+p_{f,k}+ e_f \geq 1$ implies that $f\neq \varepsilon$. Thus,
by Lemma \ref{PatrickInd},  there exists an RCP function $g$ such that  \\
-- either there exists  $b\in\Sigma$ such that 
$f(x_1,\dots,x_k)= b g(x_1,\dots,x_k)$ for all  $x_1,\dots,x_k\in \Sigma^*$, and hence $p_{f,i}= p_{g,i}$  for $i=1,\ldots,k$, and $e_f=|f(\varepsilon,\ldots,\varepsilon)| =|g(\varepsilon,\ldots,\varepsilon)| +1=e_g+1$,\\
-- or there exists $i\in\{1,\ldots, k\}$ such that   $f(x_1,\dots,x_k)= x_i g(x_1,\dots,x_k)$ for all  $x_1,\dots,x_k\in \Sigma^*$, and then $p_{f,j}= p_{g,j}$ for all $j\neq i$, $p_{f,i}= p_{g,i}+1$, and  $e_f=
e_g$.\\
In both cases  $g$ is RCP and
$p_{g,1} + p_{g,2}+ \cdots+p_{g,k}+  e_g =p_{f,1}+ p_{f,2}+ \cdots + p_{f,k}+ e_f-1$. By the induction hypothesis $g$  is of the required form and so is $f$.
\end{proof}
The rest of this section is devoted to the proof of  Lemma \ref{PatrickInd},
which shall be given after Lemma~\ref{cas:zSigma*:n}.
Lemmata  \ref{l:f epsilon:n} and  \ref{cas:aSigma*n}  below respectively extend Lemmata \ref{l:f epsilon} and \ref{l:fa=bw} to arity $k$. These two Lemmata respectively deal with conditions ($C_3$) and ($C_1$) in Lemma \ref{PatrickInd}.
%
%

 %
\begin{lemma}\label{l:f epsilon:n} 
Assume $f : (\Sigma^*)^k\to \Sigma^*$ is RCP.
If there are $u_1\neq\varepsilon$, \ldots,   $u_k\neq\varepsilon$ such that $f(u_1,\ldots,u_k)=\varepsilon$
then $f$ is constant with value $\varepsilon$.
 \end{lemma}

\begin{proof} Similar to the proof of Lemma~\ref{l:f epsilon}.
Since $|f(u_1,\ldots,u_k)|=0$ and no $|u_i|$ is null, 
equality $|f(u_1,\ldots,u_k)|=p_{f,1} |u_1|+\cdots+p_{f,k} |u_k| + e_f$
yields $p_{f,1}=\cdots=p_{f,k} = e_f$,
and hence $f$ is the constant $\varepsilon$. 
\end{proof}
We now define functions obtained by ``freezing" some arguments of a given function. Such functions will be used in the proofs of  subsequent Lemmata.
\begin{definition}[Freezing arguments]
Let $f\colon (\Sigma^*)^k\to\Sigma^*$.
If $u, v_1,\ldots,v_{i-1}, v_{i+1},\ldots, v_k\in\Sigma^*$ and $i\in \{1,\ldots,k\}$
we denote by  $f^{v_1,\ldots,v_{i-1}, v_{i+1},\ldots, v_k}_i$ 
and $f^u_{1,\ldots,i-1,i+1,\ldots,k}$
the unary and $k-1$ ary functions over $\Sigma^*$ such that, 
for all $x, x_1,\ldots,x_{i-1},u,x_{i+1},\ldots,x_k\in\Sigma^*$,
\begin{eqnarray*}
f^{v_1,\ldots,v_{i-1}, v_{i+1},\ldots, v_k}_i(x)
&=& f(v_1,\ldots,v_{i-1},x,v_{i+1},\ldots,v_k)
\\
f^u_{1,\ldots,i-1,i+1,\ldots,k}(x_1,\ldots,x_{i-1},x_{i+1},\ldots,x_k)
&=& f(x_1,\ldots,x_{i-1},u,x_{i+1},\ldots,x_k)
\end{eqnarray*}
\end{definition}
\begin{lemma}\label{l:proj:n} 
Let $|\Sigma|\geq 3$. If $f\colon (\Sigma^*)^k\to\Sigma^*$ is RCP 
then  the functions $f^{v_1,\ldots,v_{i-1}, v_{i+1},\ldots, v_k}_i$
and $f^u_{1,\ldots,i-1,i+1,\ldots,k}$ are also RCP for all
$i\in \{1,\ldots,k\}$ and $u, v_1,\ldots,v_{i-1}, v_{i+1},\ldots, v_k\in\Sigma^*$.
 \end{lemma}
 \begin{proof} Straightforward.
 \end{proof}

 \begin{lemma}\label{cas:aSigma*n} 
Let $|\Sigma|\geq 3$ and $f\colon (\Sigma^*)^k\to\Sigma^*$ be RCP.  Assume  there exists $b\in\Sigma$ and $(x_1,\ldots,x_k)\in (\Sigma^*)^{k}$ such that none of $x_1,\ldots,x_k$ has $b$ as first letter and $f(x_1,x_2,\ldots,x_k) \in b\Sigma^*$. Then $f(z_1,z_2,\ldots,z_k) \in b\Sigma^*$  for all   $(z_1,z_2,\ldots,z_k)\in (\Sigma^*)^{k}$.
 \end{lemma}
\begin{proof}  By induction on $k$. {\it Base case: $k=1$.} It follows from  Lemma~\ref{Patrick}. \\
 {\it Induction.} Let $k\geq 2$ and assume the Lemma holds for arities $n<k$. Let $\vec x=
(x_2,\ldots,x_k)$.  
The function $f^{\vec x}_{1}$ is RCP by Lemma \ref{l:proj:n}. As $x_1\not\in b\Sigma^*$ and $f^{\vec x}_{1}(x_1)=f(x_1,x_2,\ldots,x_k)\in b\Sigma^*$  Lemma~\ref{Patrick} implies that $f^{\vec x}_1(z_1)\in b\Sigma^*$  for all $z_1\in\Sigma^*$, and hence also 
  $f^{z_1}_{2\ldots k}(\vec x)=f(z_1,\vec x)=f^{\vec x}_{1}(z_1)\in b\Sigma^*$. As $f^{z_1}_{2\ldots k}$ is a $(k-1)$-ary RCP function by Lemma \ref{l:proj:n}, and as the first letters of $x_2,\ldots,x_k$ are all different of $b$,  the induction hypothesis
insures that $f^{z_1}_{2\ldots k}(\vec z)\in b\Sigma^*$ for all $\vec z=(z_2,\ldots,z_k)\in (\Sigma^*)^{k-1}$.   As $f(z_1,z_2,\ldots,z_k)=f^{z_1}_{2\ldots k}(\vec z)$, we have  for all $z_1, \vec z$,  $f(z_1,z_2,\ldots,z_k)\in b\Sigma^*$ and the induction is proved.
\end{proof}

 \begin{lemma}\label{cas:zSigma*:n} 
Let $|\Sigma|\geq 3$ and let $f\colon (\Sigma^*)^k\to\Sigma^*$ be RCP.
Assume the two following conditions hold
\begin{enumerate}
\item[(*)]
For every $i\in\{1,\ldots,k\}$ and $u\in\Sigma^*$,
the function $f^u_{1,\ldots,i-1,i+1,\ldots,k}$ satisfies exactly one  of the 
conditions ($C_1$), ($C_2$), ($C_3$) of Lemma~\ref{PatrickInd}.
\item[(**)]
There exists $y\in\Sigma^*\setminus\{\varepsilon\}$ and an index $i\in\{2,\ldots,k\}$ 
such that 
$f(y,x_2,\ldots,x_k) \in x_i\Sigma^*$ for all $(x_2,\ldots,x_k)\in (\Sigma^*)^{k-1}$.
\end{enumerate}
Then $f(z_1,z_2,\ldots,z_k) \in z_i\Sigma^*$
for all $(z_1,z_2,\ldots,z_k)\in (\Sigma^*)^k$. \end{lemma}
\begin{proof}   For all $y'\in\Sigma^*$ 
the function $f^{y'}_{2\ldots k}$ is RCP.    Hence by condition (*), there are exactly three possible cases, one of which splits into two subcases. The proof idea is as follows: for all $y'\in\Sigma^*$, all cases lead to a contradiction except subcase 3.2. Finally, stating that subcase 3.2 holds for all  $y'\in\Sigma^*$ is  exactly the conclusion of the Lemma.\\
Let $i$ and $y$ be as given in condition (**). Condition (**) implies 
\begin{eqnarray}\label{f^ay-in-aSigma*}
\text{for all } \vec a=(a,\ldots,a)\in\Sigma^{k-1}\qquad\qquad\quad f^y_{2\ldots k}(\vec a) = f^{\vec a}_{1}(y)=f(y,\vec a) &\in& a\Sigma^*\qquad\qquad \qquad 
\end{eqnarray}
\smallskip
{\bf Case 1.}  Condition ($C_1$) holds for $y'$: {\it for some $b\in\Sigma$,   
$f^{y'}_{2\ldots k}(\vec x) \in b\Sigma^*$ for all $\vec x\in(\Sigma^*)^{k-1}$.}
We show that this case is impossible.
Since $|\Sigma|\geq3$, there exists $a\in\Sigma$ which is different from $b$ 
and from the first letter of $y$.
Set $\vec a=(a,\ldots,a)\in \Sigma^{k-1}$. {\bf Case 1} hypothesis implies
\begin{eqnarray}\label{f^ay-in-bSigma*}
   f^{\vec a}_{1}(y')=f^{y'}_{2\ldots k}(\vec a) &\in& b\Sigma^*
\end{eqnarray}
As $f^{\vec a}_{1}$ is a unary RCP function, it has one of the three forms given in 
Lemma~\ref{Patrick}. 
Conditions \eqref{f^ay-in-aSigma*} and 
\eqref{f^ay-in-bSigma*}  show that $f^{\vec a}_{1}$ can only be of the form 
$z\mapsto zg(z)$ for all $z$. 
Applying condition \eqref{f^ay-in-bSigma*}, we see that the first letter of $y$ should be $a$, 
contradicting the choice of $a$.
\\
{\bf Case 2.} Condition ($C_3$) holds for $y'$: {\it   $f^{y'}_{2\ldots k}(\vec x) =\varepsilon$ for all $\vec x\in \Sigma^{k-1}$.} \\
This  case also is excluded.
Let $a$ be different from the first letter of $y$ and $\vec a=(a,\ldots,a)\in \Sigma^{k-1}$.  By condition \eqref{f^ay-in-aSigma*},  $f_{1}^{\vec a}$ is not the constant function $\varepsilon$.  Thus, by Lemma \ref{Patrick} the RCP function  
$f_{1}^{\vec a}$ can have two possible forms

-- Either  there is some $b\in\Sigma$ such that $f_{1}^{\vec a}(z)\in b\Sigma^*$ for all $z$,
in particular $f_{1}^{\vec a}(y')\in b\Sigma^*$. As
 $f_{1}^{\vec a}(y')= f^{y'}_{2\ldots k}(\vec a)=f(y',\vec a)$,   {\bf Case 2.} hypothesis implies $f_{1}^{\vec a}(y') =\varepsilon$, a contradiction.

-- Or $f_{1}^{\vec a}(z)\in z\Sigma^*$ for all $z$, which,  letting $z=y$, implies
$f_{1}^{\vec a}(y)\in y\Sigma^*$.
 This contradicts condition \eqref{f^ay-in-aSigma*} because $a$ is assumed to be different from the first letter of $y$.
\\
{\bf Case 3.} Condition ($C_2$) holds for $y'$: {\it there exists an index $j\in\{2,\ldots,k\}$ such that  
$f^{y'}_{2\ldots k}(z_2,\ldots,z_k) \in z_j\Sigma^*$ for all 
$(z_2,\ldots,z_k)\in(\Sigma^*)^{k-1}$.}  This case naturally splits into two subcases.
\\
\smallskip\noindent
{\it Subcase 3.1.} If $j\neq i$ we deduce a contradiction by letting  $a\in\Sigma$ be different from the first letter of $y$ and considering the $(k-1)$-ary RCP function  $f^{a}_{1\ldots (i-1)  (i+1)\ldots k}$. Noting that $f^{a}_{1\ldots (i-1)  (i+1)\ldots k}(z_1,z_2,\ldots,z_{i-1},z_{i+1},\ldots,z_k)=f(z_1,z_2,\ldots,z_{i-1},a, z_{i+1},\ldots,z_k) $ for all  $(k-1)$-tuple
$\vec{z} = (z_1,z_2,\ldots,z_{i-1},z_{i+1},\ldots,z_k)$, we have 
\begin{eqnarray}\label{eq:z1= y}
\text{By condition (**)} \qquad\qquad z_1=y \ \Longrightarrow \ f^{a}_{1\ldots (i-1)  (i+1)\ldots k}(\vec{z})&\in& a\Sigma^* 
\\\label{eq:z1=yprime}
\text{By {\bf Case 3.} hypothesis}  \qquad\quad z_1=y'\ \Longrightarrow \ f^{a}_{1\ldots (i-1)  (i+1)\ldots k}(\vec{z})&\in& z_j\Sigma^*
\end{eqnarray}
By condition (*) the function $f^{a}_{1\ldots (i-1)  (i+1)\ldots k}$  has three possible  forms. 
\begin{enumerate}
\item[($C_1$)] Either for some $b\in\Sigma$,
$f^{a}_{1\ldots (i-1)  (i+1)\ldots k}(\vec{z}) \in b\Sigma^*$ for every $\vec{z}\in(\Sigma^*)^{k-1}$. It is excluded:
it contradicts \eqref{eq:z1=yprime} when $z_1=y'$ 
and the first letter of $z_j$ is different from $b$.

\item[($C_2$)] Or, for some $\ell\in\{1,\ldots,i-1,i+1,\ldots,k\}$, 
we have $f^{a}_{1\ldots (i-1)  (i+1)\ldots k}(\vec{z}) \in z_\ell\Sigma^*$ 
for every $\vec{z}\in(\Sigma^*)^{k-1}$. This contradicts \eqref{eq:z1= y} when $z_1=y$.  
Indeed, if $\ell\neq 1$ then  the first letter of $z_\ell$ can be chosen different from $a$, and if 
$\ell= 1$ then the first letter of $y$ already  is different from $a$.

\item [($C_3$)] Or $f^{a}_{1\ldots (i-1)  (i+1)\ldots k}=\varepsilon$. 
This is impossible because it contradicts \eqref{eq:z1= y}.
\end{enumerate}
{\it Subcase 3.2.} If $j=i$ then $f(y',z_2,\ldots,z_k)=f^{y'}_{2\ldots k}(z_2,\ldots,z_k)\in z_i\Sigma^*$
for all $z_2,\ldots,z_k\in\Sigma^*$. For all $y'$ this is the only non contradictory case, and hence  $f(y',z_2,\ldots,z_k)\in z_i\Sigma^*$  for all $y', z_2,\ldots,z_k$, and the conclusion of the Lemma holds.
 \end{proof}
 We finally prove Lemma \ref{PatrickInd}.
\begin{proof}[Proof of Lemma \ref{PatrickInd}]
We argue by induction on the arity $k$. \\
{\it Basis.} For $k=1$ this is Lemma \ref{Patrick}. \\
{\it Induction.} Let $k\geq 2$.
 Assume the  result holds for arity $<k$, we prove  that it also holds for arity $k$. 
 Consider the $(k-1)$-ary functions $f^a_{2,\ldots,k}$, for $a\in\Sigma$. They are  RCP  by Lemma~\ref{l:proj:n}. By the induction hypothesis they must satisfy exactly one of the conditions $C_1,\ C_2, \ C_3$, and hence
\begin{enumerate}
\item Either {\it there exists $a\in\Sigma$ such that $f^a_{2,\ldots,k}$ satisfies ($C_1$)
relative to some $b\neq a$,}\/ i.e.,  $f^a_{2,\ldots,k}(x_2,\ldots,x_k) \in b\Sigma^*$ for all 
$x_2,\ldots,x_k\in\Sigma^*$. By Lemma~\ref{cas:aSigma*n}, $f$ satisfies ($C_1$).

\item or {\it  for all $a\in\Sigma$ the function $f^a_{2,\ldots,k}$ satisfies ($C_1$) relative to $b=a$,\/} 
i.e., for all $x_2,\ldots,x_k\in\Sigma^*$, $f^a_{2,\ldots,k}(x_2,\ldots,x_k) \in a\Sigma^*$.
For any $u_2,\ldots,u_k\in\Sigma^*$,  the unary function $f^{u_2,\ldots,u_k}_1$ satisfies $f^{u_2,\ldots,u_k}_1(a) \in a\Sigma^*$ for any $a\in\Sigma$.
By Lemma~\ref{l:fx=xw},  $f^{u_2,\ldots,u_k}_1(u) \in u\Sigma^*$ for any $u\in\Sigma^*$.
As this holds for any $u_2,\ldots,u_k\in\Sigma^*$, we infer that $f$ satisfies ($C_2$)
relative to the index $1$.

\item Or {\it there exists $a\in\Sigma$ such that $f^a_{2,\ldots,k}$ satisfies ($C_2$) 
relative to an index $i\in\{2,\ldots,k\}$,}
i.e., $f^a_{2,\ldots,k}(x_2,\ldots,x_k) \in x_i\Sigma^*$ for all $x_2,\ldots,x_k\in\Sigma^*$.
Then condition~(**) of Lemma~\ref{cas:zSigma*:n} holds.
As the arity of $f^a_{2,\ldots,k}$ is $(k-1)$,
the induction hypothesis insures that condition~(*) of Lemma~\ref{cas:zSigma*:n}  also holds.
Applying Lemma~\ref{cas:zSigma*:n}, we see that $f$ satisfies ($C_2$) with the same index $i$.

\item or {\it there exists $a\in\Sigma$ such that $f^a_{2,\ldots,k}$ satisfies ($C_3$)}.
Then $f(a,a,\ldots,a)=\varepsilon$ and Lemma~\ref{l:f epsilon:n} shows that $f$ satisfies ($C_3$).
\qedhere
 \end{enumerate}
\end{proof}
\section{ $k$-ary case with infinite alphabet}
The proof of the passage to arity $k$ is much simpler if the alphabet $\Sigma$ is infinite
rather than of cardinality at least $3$.
\begin{theorem}
Assume $\Sigma$ is infinite. Then every RCP function $f:(\Sigma^*)^k \to \Sigma^*$
is of the form $f(x_1,\ldots,x_k)=w_0x_{i_1}^{p_1}w_1x_{i_2}^{p_2}w_2\cdots x_{i_n}^{p_n}w_n$,
where $x_{i_j}\in\{x_1,\ldots,x_k\}$ for $j=1,\ldots,n$.
\end{theorem}

\begin{proof}
We argue by induction on the arity $k$.

{\it Base case $k=1$.} This is Theorem~\ref{Carac_par_Poly RCP}.

{\it Induction.} Let $k\geq2$.
We assume the theorem is true for arity $k-1$ and we prove it for arity $k$.
Fix some $\vec{x}=(x_2,\ldots,x_k) \in (\Sigma^*)^{k-1}$.
The unary function $f^{\vec{x}}_1 : \Sigma^* \to \Sigma^*$,
obtained from $f$ by freezing all arguments but the first one, is RCP.
Also, with the notations of Lemma~\ref{l:lengthK}, we have, for all $x_1\in\Sigma^*$,
\begin{eqnarray}\label{eq:length fvecx}
|f^{\vec{x}}_1(x_1)| = |f(x_1,x_2,\ldots,x_k)|
= m |x_1| + n
\end{eqnarray}
where $m=  p_{f,1}$ and $n = p_{f,2} |x_2|+\cdots+ p_{f,k} |x_k| +e_f$.

Since the unary function $f_1^{\vec{x}}$ is RCP, 
applying Theorem~\ref{Carac_par_Poly RCP}, Lemma~\ref{l:length} 
and equation \eqref{eq:length fvecx}, 
we see that there exists $m+1$ words $u_0(\vec{x}),\ldots, u_m(\vec{x})$ 
(which depend only on $\vec{x}$)
such that, for all $x_1$ and $\vec{x}$,
\begin{eqnarray}\label{eq:fxt}
f^{\vec{x}}_1(x_1) \ =\ f(x_1,x_2,\ldots,x_k) &=& 
u_0(\vec{x})\, x_1\, u_1(\vec{x})\, x_1 \cdots u_{m-1}(\vec{x})\, x_1\, u_m(\vec{x})
\end{eqnarray}
{\bf Claim.} {\it The functions 
$\vec{x}\mapsto u_0(\vec{x}),  \ldots,  \vec{x}\mapsto u_m(\vec{x})$ are RCP.}
\smallskip\\
{\it Proof of Claim.}
Let $\varphi:\Sigma^*\to\Sigma^*$ be a morphism and 
$\vec{y}=(y_2,\ldots,y_k), \vec{z}=(z_2,\ldots,z_k)$ in $(\Sigma^*)^{k-1}$ be such that
$\varphi(y_2)=\varphi(z_2)$, \ldots, $\varphi(y_k)=\varphi(z_k)$.
We have to prove that $\varphi(u_i(\vec{y}))=\varphi(u_i(\vec{z}))$ for $i=0,\ldots,m$.

Let $\Gamma$ be a finite subset of $\Sigma$ such that 
$y_2$,\ldots, $y_k$, $z_2$,\ldots, $z_k$, 
$u_0(\vec{y})$,\ldots, $u_m(\vec{y})$, $u_0(\vec{z})$,\ldots, $u_m(\vec{z})$ 
and their images by $\varphi$
are all in $\Gamma^*$.
Let $a\in\Sigma\setminus\Gamma$ 
(this is where we use the hypothesis that $\Sigma$ is infinite).
Define a morphism $\psi: \Sigma^*\to\Sigma^*$ as follows:
$\psi(c)=\varphi(c)$ for all $c\in\Gamma$ and $\psi(c)=a$ for all $c\in\Sigma\setminus\Gamma$.
In particular, we have 
$\psi(u_i(\vec{y}))=\varphi(u_i(\vec{y}))$ and $\psi(u_i(\vec{z}))=\varphi(u_i(\vec{z}))$
for $i=0,\ldots,m$,
and these words contain no occurrence of $a$.
Thus, applying the morphism $\psi$ to \eqref{eq:fxt} with 
$x_1=a$ and $\vec{x}=\vec{y},\vec{z}$, we get
\begin{eqnarray}\notag
\psi(f(a,\vec{y})) &=& \psi(u_0(\vec{y}))\, \psi(a)\, \psi(u_1(\vec{y}))\, \psi(a) \cdots 
\psi(u_{m-1}(\vec{y}))\, \psi(a)\, \psi(u_m(\vec{y}))
\\\label{eq:psi fay}
&=& \varphi(u_0(\vec{y}))\, a\, \varphi(u_1(\vec{y}))\, a \cdots 
\varphi(u_{m-1}(\vec{y}))\, a\, \varphi(u_m(\vec{y}))
\\\label{eq;psi faz}
\text{Similarly,\ }
\psi(f(a,\vec{z})) &=& \varphi(u_0(\vec{z}))\, a\, \varphi(u_1(\vec{z}))\, a \cdots 
\varphi(u_{m-1}(\vec{z}))\, a\, \varphi(u_m(\vec{z}))
\end{eqnarray}
As $\psi(y_2)=\varphi(y_2)=\varphi(z_2)=\psi(z_2)$,\ldots,
$\psi(y_k)=\varphi(y_k)=\varphi(z_k)=\psi(z_k)$
and $f$ is RCP, we have $\psi(f(a,\vec{y}))=\psi(f(a,\vec{z})$.
Applying equations \eqref{eq:psi fay} and \eqref{eq;psi faz}, we get
\begin{multline}\label{eq: on separe avec a}
\varphi(u_0(\vec{y}))\, a\, \varphi(u_1(\vec{y}))\, a \cdots 
\varphi(u_{m-1}(\vec{y}))\, a\, \varphi(u_m(\vec{y}))
\\
=\  \varphi(u_0(\vec{z}))\, a\, \varphi(u_1(\vec{z}))\, a \cdots 
\varphi(u_{m-1}(\vec{z}))\, a\, \varphi(u_m(\vec{z}))
\end{multline}
Since $a$ does not occur in the $\varphi(u_i(\vec{y}))$'s and the $\varphi(u_i(\vec{z}))$'s, 
for $i=0,\ldots,m$, equality \eqref{eq: on separe avec a} yields
$\varphi(u_0(\vec{y}))=\varphi(u_0(\vec{z}))$, \ldots, 
$\varphi(u_m(\vec{y}))=\varphi(u_m(\vec{z}))$.
This proves the Claim.

\medskip
Finally, applying the induction hypothesis, 
the RCP $(k-1)$-ary functions $u_0(\vec{x}),\ldots, u_m(\vec{x})$
are represented by terms in $x_2,\ldots,x_k$.
Using equation~\eqref{eq:fxt}, we then get a term 
which represents the $k$ ary function $f(x_1,x_2,\ldots,x_k)$.
\end{proof}

\section{Conclusion}
We proved that, when $\Sigma$ has at least three letters, the free monoid $\Sigma^*$ is affine complete, i.e., a function is CP if and only if  it is ``polynomial". An essential tool in the proof was to use restricted congruence preserving functions, which happen to coincide with CP functions in case $\Sigma$ has at least three letters.

If $\Sigma$ has just one letter,  RCP functions are a strict subset of CP functions
because the monoid $\Sigma^*$ then reduces to the monoid $\langle\N,+\rangle$ which has non restricted congruences non equivalent to restricted ones.
We proved in \cite{cggIJNT} that  there are on $\langle\N,+\rangle$ non polynomial CP functions, e.g.,
$f(x)\ =\ \lfloor e^{1/a} a^x   x!\rfloor$ for $a\in\N\setminus\{0,1\}$.

An open problem is to characterize CP functions when $\Sigma$ has exactly two letters: 
is $\{a,b\}^*$  affine complete or are there non polynomial CP functions, i.e., 
do Theorems~\ref{Carac_par_Poly} and \ref{t:main} extend to  binary alphabets?

\bibliographystyle{plain}

\end{document}